\documentclass[11pt]{article}
\usepackage[margin=1in]{geometry}
\usepackage{amsmath,amssymb,amsthm,mathtools}
\usepackage{microtype}
\usepackage{enumitem}
\usepackage[colorlinks=true,linkcolor=blue,citecolor=blue,urlcolor=blue]{hyperref}
\usepackage[nameinlink,capitalise,noabbrev]{cleveref}
\newtheorem{theorem}{Theorem}[section]
\newtheorem{lemma}[theorem]{Lemma}
\newtheorem{proposition}[theorem]{Proposition}
\newtheorem{corollary}[theorem]{Corollary}

\theoremstyle{definition}

\newcommand{\R}{\mathbb R}
\newcommand{\e}{\mathrm e}
\newcommand{\kk}{\mathbf k}
\newcommand{\uu}{\mathbf u}
\newcommand{\ssv}{\mathbf s}
\newcommand{\cc}{\mathbf c}
\newcommand{\qq}{\mathbf q}
\newcommand{\ee}{\mathbf e}
\newcommand{\Ent}{H}
\newcommand{\DG}{\Delta}
\newcommand{\pot}{\psi}
\newcommand{\KL}{D_{\mathrm{KL}}}
\title{An Improved Upper Bound for Multicolour Ramsey Numbers}
\author{%
 Sunghyeon Jo\\
 {\small Georgia Institute of Technology and QED Audit}\\
 {\small\texttt{sjo65@gatech.edu}}
}
\date{\today}
\begin{document}
\maketitle

\begin{abstract}
Let $R_r(k)$ denote the diagonal $r$-colour Ramsey number.  We prove that
there exist absolute constants $c,K>0$ such that
\[
 R_r(k)\le r^{rk}\exp\!\left(-c\frac{k}{r\log^2(2r)}\right)
\]
for every $r\ge2$ and every $k\ge Kr^2\log^6(2r)$.  This improves the
exponential saving in a recent bound of Yang and Mao by a factor of order
$r\log^2(2r)$.  The proof proceeds through an off-diagonal bound, which
asymptotically improves the classical multinomial bound throughout a
neighbourhood of the diagonal.
\end{abstract}

\section{Introduction}
For integers $r,k\ge2$, let $R_r(k)$ denote the least $n$ such that every
$r$-colouring of the edges of $K_n$ contains a monochromatic copy of $K_k$.
The classical Erd\H{o}s--Szekeres argument \cite{ES} gives
\[
 R_r(k)\le r^{rk}.
\]
For two colours, Campos, Griffiths, Morris and Sahasrabudhe proved the first
exponential improvement over the Erd\H{o}s--Szekeres bound
\cite{CGMS}.  A shorter book-based proof, which extends to every fixed number
of colours, was subsequently obtained by Balister, Bollob\'as, Campos,
Griffiths, Hurley, Morris, Sahasrabudhe and Tiba \cite{BBCGHMST}.  In
particular, for every $r\ge2$ their explicit estimate saves a factor
$\exp(\Omega(k/r^{12}))$ over $r^{rk}$ once $k$ is at least of order
$r^{20}$.  The best known lower bounds are also exponential in $k$ for every
fixed $r$, but a large gap between the lower and upper exponential rates
remains \cite{CamposPohoata}.

Recent work has improved the dependence on the number of colours.
Narang and Tang obtained a saving of order $k/(r^9\log^6(2r))$, for $k$ at
least of order $r^{14}\log^{12}(2r)$, using robust OR polynomials
\cite{NarangTang}.  Yang and Mao
then introduced a variable-order positive root filter and combined it with a
retained-spine refinement of the multicolour book method \cite{YangMao}.  Their
main bound is
\begin{equation}\label{eq:YM-main}
 R_r(k)\le r^{rk}\exp\!\left(-\Omega\!\left(\frac{k}{r^2\log^4(2r)}\right)\right)
\end{equation}
for $k$ at least of order $r^2\log^6(2r)$.
These two scales have different origins.  Recall that a monochromatic book
consists of a monochromatic clique, its spine, together with a set of common
neighbours in the same colour, its page set.  The root-filter argument produces
a spine of relative size
\[
 \theta\asymp \frac1{r\log^2(2r)},
\]
whereas the final step, an application of the classical multinomial bound to
the page set, contributes only $\exp(-\Omega(\theta^2k))$.

Our main observation is that the latter loss is not intrinsic to the book
construction.  Instead of applying the multinomial bound once to the page
set, we retain the full off-diagonal target vector and apply the same argument
recursively.  This yields a saving of order $\theta k$ rather than
$\theta^2k$.  Gupta, Ndiaye, Norin and Wei~\cite{GNNW} also use off-diagonal
Ramsey bounds in an inductive argument that extends to the multicolour
setting.  Here we apply the induction to the page produced by the Yang--Mao
multicolour book construction.

\begin{theorem}\label{thm:main}
There exist absolute constants $c,K>0$ such that, for every $r\ge2$ and every
\[
 k\ge K r^2\log^6(2r),
\]
one has
\[
 R_r(k)\le
 r^{rk}\exp\!\left(-c\frac{k}{r\log^2(2r)}\right).
\]
\end{theorem}

The saving in \cref{thm:main} improves that in \eqref{eq:YM-main} by a factor
of order $r\log^2(2r)$, while retaining the same range of $k$, up to
absolute constants.  These comparisons
concern a growing number of colours.  In particular, for $r=2$, the
specialised bounds of \cite{CGMS,GNNW} are much stronger.

The proof is based on a stronger off-diagonal statement.  For positive
integers $k_1,\ldots,k_r$, write $[r]=\{1,\ldots,r\}$,
$\kk=(k_1,\ldots,k_r)$ and $(x)_+=\max\{x,0\}$, and let $R(k_1,\ldots,k_r)$
be the least $n$ such that every $r$-colouring of $K_n$ contains a
colour-$i$ copy of $K_{k_i}$ for some $i\in[r]$.  Put
\begin{equation}\label{eq:entropy-def-intro}
 B=\sum_{i=1}^r k_i,
 \qquad
 \Ent(\kk)=B\log B-\sum_{i=1}^r k_i\log k_i,
\end{equation}
and
\[
 m=\min_i k_i,\qquad M=\max_i k_i,
 \qquad \pot(\kk)=(2m-M)_+.
\]
The function $\Ent$ is the logarithmic form of the classical multinomial bound,
while $\pot(\kk)$ measures the slack in the inequality $M<2m$.

\begin{theorem}\label{thm:offdiag}
There exist absolute constants $c,C>0$ such that, for every $r\ge2$ and every
positive integer vector $\kk=(k_1,\ldots,k_r)$,
\begin{equation}\label{eq:offdiag}
 R(\kk)\le
 \left\lceil
 \exp\!\left(
 \Ent(\kk)
 -c\frac{\pot(\kk)}{r\log^2(2r)}
 +Cr\log^4(2r)
 \right)
 \right\rceil.
\end{equation}
\end{theorem}

For the diagonal vector, $\Ent(k,\ldots,k)=rk\log r$ and
$\pot(k,\ldots,k)=k$.  The error term in \eqref{eq:offdiag} is absorbed by the
saving once $k\ge K r^2\log^6(2r)$, so \cref{thm:main} follows.  For the proof it
is convenient to keep the order $d$ of the Yang--Mao root filter as a
parameter; the corresponding statement (\cref{thm:offdiag-flex}) is given in
\cref{sec:mainproof}.

Suppose
that $s_i$ vertices have already been removed from each target $k_i$, and
that a colour-$i$ book then removes a further $t$ vertices from the $i$th
coordinate.  The multinomial exponent changes by an exact relative-entropy
term.  If the removed sizes are nearly proportional to the targets, this
term can be small, but then the gain from the regularisation step is
sufficient.  A spine of size $t$ added in a single colour, however, is far from
proportional, and the relative entropy is then
$\Omega(t^2/m)=\Omega(\theta t)$.  This is enough to apply the induction to
the page while keeping a saving of order $\theta$, and iterating turns the
one-step $\theta^2k$ gain into $\theta k$.

The proof uses the higher-order correlation theorem of Yang and Mao only
through their density-increment lemma and the resulting book construction.
We restate the density-increment lemma and prove the form of the book theorem
we need in \cref{sec:books}; the only change is that the baseline density and
the required page size are allowed to depend on the colour.  The
density-increment lemma already allows colour-dependent parameters, so no
new analytic estimate is needed.

\section{Entropy and regularisation}\label{sec:prelim}
Throughout the paper, all logarithms are natural, and implicit constants in
$O(\cdot)$, $\Omega(\cdot)$ and $\Theta(\cdot)$ are absolute unless a
subscript is displayed.  We write $\ee_i$ for the $i$th standard basis vector
of $\R^r$.  In an $r$-edge-coloured complete graph,
$N_i(v)=\{w:vw\text{ has colour }i\}$ denotes the colour-$i$ neighbourhood of
a vertex $v$.

The Erd\H{o}s--Szekeres recursion gives
\begin{equation}\label{eq:ES}
 R(k_1,\ldots,k_r)
 \le \binom{B-r}{k_1-1,\ldots,k_r-1}
 \le \binom{B}{k_1,\ldots,k_r}
 \le \exp(\Ent(\kk)).
\end{equation}
We shall repeatedly compare the exponent before and after a vector of target
sizes is removed.  For probability vectors $p,q\in\R_{>0}^r$, write
\[
 \KL(p\|q)=\sum_{i=1}^r p_i\log\frac{p_i}{q_i}.
\]

\begin{lemma}\label{lem:entropy}
Let $\kk\in\R_{>0}^r$ and let $\uu\in\R_{\ge0}^r$ satisfy
$\cc=\kk-\uu\in\R_{>0}^r$.  Put
\[
 B=\sum_i k_i,\qquad U=\sum_i u_i,
 \qquad q_i=\frac{k_i}{B}.
\]
Then
\begin{equation}\label{eq:entropy-identity}
 \Ent(\kk)+\sum_i u_i\log q_i-\Ent(\cc)
 =(B-U)\,\KL\!\left(\frac{\cc}{B-U}\middle\|\qq\right)
 =:\DG_{\kk}(\uu).
\end{equation}
Moreover, if $z_i=u_i-q_iU$, then
\begin{equation}\label{eq:entropy-lower}
 \DG_{\kk}(\uu)\ge
 \frac12\sum_{i=1}^r\frac{z_i^2}{k_i}.
\end{equation}
\end{lemma}

\begin{proof}
Expanding the left-hand side of \eqref{eq:entropy-identity} gives
\[
 \sum_i c_i\log\frac{c_i/(B-U)}{k_i/B},
\]
which is the stated relative entropy.  For the lower bound, use
\[
 x\log(x/y)-x+y\ge \frac{(x-y)^2}{2\max\{x,y\}}
 \qquad(x,y>0)
\]
with $x=c_i$ and $y=(B-U)q_i$.  Since both are at most $k_i$ and
$c_i-(B-U)q_i=-z_i$, summing over $i$ proves \eqref{eq:entropy-lower}.
\end{proof}

For $\kk\in\R_{>0}^r$, recall that
\[
 m=\min_i k_i,\qquad M=\max_i k_i,
 \qquad \pot(\kk)=(2m-M)_+.
\]
The next estimate is the reason for this particular potential.  If $\uu$ is
nearly proportional to $\kk$, then $\pot$ changes only on the scale $U/r$;
the deviation from proportionality is paid for by the relative-entropy term.

\begin{lemma}\label{lem:potential}
Let $\kk\in\R_{>0}^r$ satisfy $\pot(\kk)>0$, let
$\uu\in\R_{\ge0}^r$ and $\cc=\kk-\uu\in\R_{>0}^r$, and put
$U=\sum_i u_i$ and $m=\min_i k_i$.  For every $\Gamma>0$,
\[
 \Gamma\bigl(\pot(\kk)-\pot(\cc)\bigr)
 \le \frac{4\Gamma U}{r}
      +\frac12\DG_{\kk}(\uu)+10\Gamma^2m.
\]
\end{lemma}

\begin{proof}
Since $\pot(\kk)>0$, we have $M<2m$.  Choose $j$ with
$c_j=\min_i c_i$ and $h$ with $c_h=\max_i c_i$.  Then
\[
 \pot(\kk)-\pot(\cc)
 \le (2m-M)-(2c_j-c_h)\le2u_j-u_h.
\]
Let $q_i=k_i/B$ and $z_i=u_i-q_iU$.  Since $M<2m$ and $B\ge rm$, we have
$q_j<2/r$, and hence
\[
 2u_j-u_h\le\frac{4U}{r}+|2z_j-z_h|.
\]
By Cauchy--Schwarz and \eqref{eq:entropy-lower},
\[
 |2z_j-z_h|^2
 \le(4k_j+k_h)\sum_i\frac{z_i^2}{k_i}
 \le20m\DG_{\kk}(\uu).
\]
The result follows from
$\Gamma\sqrt{20m\DG}\le \DG/2+10\Gamma^2m$.
\end{proof}

We shall also need a weighted form of the multicolour Erd\H{o}s--Szekeres
regularisation, in which the common baseline $1/r$ is replaced by the target
proportion $q_i=k_i/B$.

\begin{lemma}\label{lem:weighted}
Let $\qq=(q_1,\ldots,q_r)$ be a probability vector with every $q_i>0$, let
$0<\eta<\min_iq_i$, and let $s_\ast$ be a positive integer.  Every
$r$-edge-coloured $K_n$ contains pairwise disjoint sets
$S_1,\ldots,S_r,W$ such that, writing $s_i=|S_i|$ and $s=\sum_i s_i$,
\begin{equation}\label{eq:weighted-size}
 |W|\ge n(1+\eta)^s\prod_{i=1}^r q_i^{s_i},
\end{equation}
$S_i$ is a colour-$i$ clique, and every edge between $S_i$ and $W$ has colour
$i$.  Moreover, either $s=s_\ast$, or
\begin{equation}\label{eq:weighted-degree}
 |N_i(w)\cap W|\ge(q_i-\eta)|W|-1
 \qquad(w\in W,\ i\in[r]).
\end{equation}
\end{lemma}

\begin{proof}
Start with $W=V(K_n)$ and all $S_i$ empty.  Suppose $s<s_\ast$ and
\eqref{eq:weighted-degree} fails at $(w,\ell)$.  Since the edges from $w$ to
$W\setminus\{w\}$ have exactly one colour,
\[
 \sum_{j\ne\ell}|N_j(w)\cap W|>(1-q_\ell+\eta)|W|.
\]
Averaging with weights $q_j/(1-q_\ell)$ shows that some $j\ne\ell$ satisfies
\[
 |N_j(w)\cap W|>
 q_j\left(1+\frac{\eta}{1-q_\ell}\right)|W|
 \ge q_j(1+\eta)|W|.
\]
Append $w$ to $S_j$ and replace $W$ by $N_j(w)\cap W$.  This preserves the
clique and attachment properties and multiplies the right-hand side of
\eqref{eq:weighted-size} by at least $q_j(1+\eta)$.  The procedure stops after
at most $s_\ast$ steps.
\end{proof}

\section{The book lemma}\label{sec:books}
We need a colour-dependent version of the Yang--Mao book lemma.
Given an $r$-edge-coloured complete graph, a colour $i\in[r]$, and nonempty
vertex sets $X,Y$, write
\[
 p_i(X,Y)=\min_{x\in X}\frac{|N_i(x)\cap Y|}{|Y|}.
\]
Thus $p_i(X,Y)$ is the minimum colour-$i$ density from $X$ to $Y$; the sets
$X$ and $Y$ need not be disjoint.  A \emph{colour-$i$ book} is an
ordered pair $(T,P)$ of disjoint vertex sets such that $T$ is a colour-$i$
clique and every edge between $T$ and $P$ has colour $i$.  We call $T$ the
\emph{spine} and $P$ the \emph{page set}.  This is the terminology used in
\cite{BBCGHMST,YangMao}.
For an integer $d\ge3$, the root-filter construction of Yang and Mao
provides parameters
\begin{equation}\label{eq:corr-params}
 \beta_r=\frac1{4r2^{r-1}(r+1)},
 \qquad C_{r,d}\le C_0rd^2\log(2rd),
\end{equation}
where $C_0$ is absolute.

\begin{lemma}[{\cite[Theorem~3.4 and Lemma~4.1]{YangMao}}]\label{lem:increment}
Let $r\ge2$ and $d\ge3$ be integers, and let $\beta_r,C_{r,d}$ be as in
\eqref{eq:corr-params}.  Let $X,Y_1,\ldots,Y_r$ be nonempty finite vertex
sets in an $r$-edge-coloured complete graph, and put
$p_i=p_i(X,Y_i)>0$ for $i\in[r]$.  For every choice of positive real numbers
$\alpha_1,\ldots,\alpha_r$ there exist a vertex $x\in X$, an index
$\ell\in[r]$, a real number $\lambda\ge-1$, a nonempty set $X'\subseteq X$,
and nonempty sets $Y_i'\subseteq N_i(x)\cap Y_i$ $(i\in[r])$ such that
\[
 |X'|\ge\beta_r\e^{-C_{r,d}(\lambda+1)^{1/d}}|X|,
 \qquad
 p_\ell(X',Y'_\ell)\ge p_\ell+\lambda\alpha_\ell,
\]
and, for every $i\in[r]$,
\[
 |Y_i'|=p_i|Y_i|,
 \qquad
 p_i(X',Y_i')\ge p_i-\alpha_i.
\]
\end{lemma}

The preceding lemma already allows a separate page set $Y_i$ and a separate
increment parameter $\alpha_i$ for each colour.  Tracking these quantities
through the proof of the Yang--Mao book theorem gives the following
colour-dependent version.

\begin{theorem}\label{thm:book}
Let $r\ge2$ and $d\ge3$ be integers, and let $\beta_r$ and $C_{r,d}$ be as in
\eqref{eq:corr-params}.  Let $t,m_1,\ldots,m_r$ be positive integers and let
$\lambda_0>0$.  For each
$i\in[r]$, let $p_i,\delta_i>0$ and define
\[
 L_i=\log(1/\delta_i),\qquad L=\max_iL_i,
 \qquad L_{p,i}=\log(1/p_i),
\]
\[
 \Pi_i=\frac{3\delta_i}{p_i}
       +\frac{6L_iL_{p,i}}{\lambda_0},
\]
and
\[
 \rho=\log(r/\beta_r),\qquad
 \Xi=2\rho+4C_{r,d}\lambda_0^{1/d}
      +\frac{12C_{r,d}L}{\lambda_0^{(d-1)/d}}.
\]
Suppose
\begin{equation}\label{eq:book-parameters}
 0<\delta_i\le\min\{p_i/4,1/4\},
 \qquad
 \lambda_0\ge\max\{2,6L\},
 \qquad
 t\ge\lambda_0(\min_i\delta_i)^{-1/(d-1)}.
\end{equation}
Let $X,Y_1,\ldots,Y_r$ be nonempty vertex sets such that
\begin{align}
 |N_i(x)\cap Y_i|&\ge p_i|Y_i|
 &&(x\in X,\ i\in[r]),\notag\\
 |Y_i|&\ge p_i^{-t}\e^{\Pi_it}m_i
 &&(i\in[r]),\notag\\
 |X|&\ge2rt\e^{rt\Xi}.\label{eq:book-reservoir}
\end{align}
Then there is a colour $i$, a colour-$i$ clique $T\subseteq X$ of size $t$,
and a set $P\subseteq Y_i\setminus T$ of size $m_i$ such that every edge
between $T$ and $P$ has colour $i$.
\end{theorem}

\begin{proof}
We follow the proof of the Yang--Mao book theorem, keeping the colour-dependent
quantities instead of replacing them by their minimum.  Let $X(s)$ be the
current reservoir, $Y_i(s)$ the current page set in colour $i$, and $T_i(s)$
the current colour-$i$ spine; initially $X(0)=X$, $Y_i(0)=Y_i$ and
$T_i(0)=\emptyset$.  Set
\[
 a_i=p_i(X(0),Y_i(0))\ge p_i,
 \qquad
 q_i(s)=p_i(X(s),Y_i(s))-a_i+\delta_i,
 \qquad
 \alpha_i(s)=\frac{q_i(s)}{t},
\]
so that $q_i(0)=\delta_i$.  While every spine has size less than $t$, apply
\cref{lem:increment} to the current sets with the positive parameters
$\alpha_i(s)$, obtaining $x$, $\ell$, $\lambda$, $X'$ and $Y_i'$.  If
$\lambda\le\lambda_0$, we \emph{extend a spine}: the edges from $x$ to
$X'\setminus\{x\}$ have one of $r$ colours, so some $j\in[r]$ satisfies
$|N_j(x)\cap X'|\ge(|X'|-1)/r$; we append $x$ to $T_j$, replace $X(s)$ by
$N_j(x)\cap X'$, and replace $Y_j(s)$ by $Y_j'$.  If $\lambda>\lambda_0$, we
\emph{boost} colour $\ell$: we replace $X(s)$ by $X'$ and $Y_\ell(s)$ by
$Y'_\ell$.  In either case all other page sets, and all spines other than the
extended one, are left unchanged.  As in \cite{YangMao}, these rules keep
each $T_i$ a colour-$i$ clique whose vertices are joined in colour $i$ to
every vertex of $X(s)\cup Y_i(s)$.
Restricting the first argument of a minimum density cannot decrease it.
Hence, if colour $i$ is extended at step $s$, then
\[
 q_i(s+1)\ge(1-1/t)q_i(s),
\]
a boost in colour $i$ with parameter $\lambda$ gives
\[
 q_i(s+1)\ge(1+\lambda/t)q_i(s),
\]
and in every other case $q_i(s+1)\ge q_i(s)$, because $Y_i$ is unchanged
while the reservoir shrinks.
Let $\mathcal B_i(s)$ be the set of boost steps in colour $i$ before time
$s$.  Before any spine reaches size $t$, at most $t$ steps extend colour $i$,
so iterating these one-step estimates, and using
$(1-1/t)^t\ge\tfrac14$ (valid since $t\ge\lambda_0\ge2$), gives
\[
 q_i(s)\ge
 \delta_i(1-1/t)^t
 \prod_{\sigma\in\mathcal B_i(s)}(1+\lambda(\sigma)/t)
 \ge \frac{\delta_i}{4}.
\]
Since every relative density is at most $1$, we also have
$q_i(s)\le1-a_i+\delta_i\le\tfrac54$, so the product above is at most
$5/\delta_i$; as $\delta_i\le\tfrac14$ gives $\log(5/\delta_i)\le3L_i$, it
follows that
\[
 \sum_{\sigma\in\mathcal B_i(s)}\log(1+\lambda(\sigma)/t)\le3L_i.
\]
Every boost counted here has $\lambda(\sigma)>\lambda_0$, and
$\lambda_0\le t$ by \eqref{eq:book-parameters}; hence
\[
 |\mathcal B_i(s)|\le \frac{6L_i}{\lambda_0}t.
\]
In particular, the bound $q_i(s)\ge\delta_i/4$ shows that the current minimum
density in colour $i$ is at least $a_i-3\delta_i/4\ge p_i-3\delta_i/4$.
We next bound the loss from the page sets.  A colour-$i$ spine extension or
boost replaces $Y_i$ by a colour-$i$ neighbourhood whose relative size is at
least $p_i-3\delta_i/4$.  There are at most $t$ spine extensions and at most
$6L_it/\lambda_0$ boosts in colour $i$.
Writing $\xi_i=3\delta_i/(4p_i)\le3/16$ and using
$\log(1-\xi_i)\ge-2\xi_i$, we obtain
\begin{align*}
 |Y_i(s)|
 &\ge p_i^{t+6L_it/\lambda_0}(1-\xi_i)^{2t}|Y_i(0)|\\
 &\ge p_i^t
 \exp\!\left(-\frac{6L_iL_{p,i}}{\lambda_0}t
              -\frac{3\delta_i}{p_i}t\right)|Y_i(0)|\\
 &=p_i^t\e^{-\Pi_it}|Y_i(0)|\ge m_i.
\end{align*}
It remains to check that the reservoir does not become empty.  Set
\[
 \varepsilon_0=\frac{\beta_r}{r}\,\e^{-C_{r,d}(\lambda_0+1)^{1/d}}.
\]
At a spine extension, $\lambda\le\lambda_0$, so \cref{lem:increment} and the
choice of the majority colour give
\[
 |X(s+1)|\ge\frac{|X'|-1}{r}
 \ge\frac{\beta_r}{r}\,\e^{-C_{r,d}(\lambda+1)^{1/d}}|X(s)|-1
 \ge\varepsilon_0|X(s)|-1,
\]
while a boost step gives
$|X(s+1)|\ge\beta_r\e^{-C_{r,d}(\lambda(s)+1)^{1/d}}|X(s)|
\ge\varepsilon_0\,\e^{-C_{r,d}(\lambda(s)+1)^{1/d}}|X(s)|$.

We next bound $\sum_\sigma(\lambda(\sigma)+1)^{1/d}$ over the boost steps.
After a boost in colour
$i$ the minimum density is still at most $1$, so
$\lambda(s)\alpha_i(s)\le1$; since $\alpha_i(s)=q_i(s)/t\ge\delta_i/(4t)$,
this gives $\lambda(s)\le4t/\delta_i$.  We claim that every threshold
$\lambda>\lambda_0$ arising at a boost step satisfies
\[
 \frac{(\lambda+1)^{1/d}}{\log(1+\lambda/t)}
 \le\frac{4t}{\lambda_0^{(d-1)/d}}.
\]
If $\lambda\le t$, then $\log(1+\lambda/t)\ge\lambda/(2t)$ and
$(\lambda+1)^{1/d}\le2\lambda^{1/d}$, so the ratio is at most
$4t\lambda^{-(d-1)/d}\le4t\lambda_0^{-(d-1)/d}$.  If $\lambda>t$, then
$\log(1+\lambda/t)\ge\log2$ and $\lambda+1\le5t/\delta_i$; since
\eqref{eq:book-parameters} gives $t\ge\lambda_0\delta_i^{-1/(d-1)}$, which
is equivalent to $t^{1/d}\delta_i^{-1/d}\le t\lambda_0^{-(d-1)/d}$, and
$5^{1/d}\le4\log2$ for $d\ge3$, the claim follows in this case as well.
Multiplying the claim by $\log(1+\lambda(\sigma)/t)$ and summing, using
$\sum_{\sigma\in\mathcal B_i(s)}\log(1+\lambda(\sigma)/t)\le3L_i$ for each
colour, we obtain
\[
 \sum_{\sigma}(\lambda(\sigma)+1)^{1/d}
 \le\frac{12rLt}{\lambda_0^{(d-1)/d}},
\]
the sum running over all boost steps.

Up to and including the first time a spine reaches size $t$, there are fewer
than $rt$ spine extensions, and the total number of boosts is at most
$\sum_i6L_it/\lambda_0\le rt$ because $\lambda_0\ge6L$.  Each step has the
form $|X(s+1)|\ge a_s|X(s)|-\epsilon_s$ with $0<a_s\le1$, where
$\epsilon_s=1$ at spine extensions and $\epsilon_s=0$ at boosts.  Since
there are at most $2rt$ steps,
\[
 \prod_qa_q
 \ge\varepsilon_0^{2rt}
 \exp\Bigl(-C_{r,d}\sum_{\sigma}(\lambda(\sigma)+1)^{1/d}\Bigr)
 \ge\e^{-rt\Xi},
\]
because $\log(1/\varepsilon_0)=\log(r/\beta_r)+C_{r,d}(\lambda_0+1)^{1/d}
\le\rho+2C_{r,d}\lambda_0^{1/d}$.  Iterating the recurrence and using
\eqref{eq:book-reservoir} therefore gives
\[
 |X(s)|\ge\Bigl(\prod_qa_q\Bigr)|X(0)|
 -\sum_j\epsilon_j\prod_{q>j}a_q
 \ge\e^{-rt\Xi}|X(0)|-rt\ge2rt-rt=rt>0.
\]
Thus the reservoir never empties and \cref{lem:increment} is applicable at
every step; since each step is a spine extension or a boost and their
numbers are bounded as above, the process performs at most $2rt$ steps, and
some spine reaches size $t$.  For that colour $i$, take $T=T_i$ and let $P$ consist of
$m_i$ vertices of the final page set $Y_i(s)$, which is possible since
$|Y_i(s)|\ge m_i$.  Every vertex of $Y_i(s)$ is a common colour-$i$
neighbour of $T$, and, the graph being simple, $Y_i(s)$ is disjoint from
$T$.  Thus $(T,P)$ is the required book.
\end{proof}

\section{Proof of the Ramsey bounds}\label{sec:mainproof}
Fix absolute constants in the following order.  Choose $A$ sufficiently
large, then $\zeta>0$ sufficiently small, then $\gamma>0$ sufficiently small,
so that
\begin{equation}\label{eq:constant-choice}
 130\gamma+20\gamma^2+3\zeta+\frac6A<\frac1{576},
 \qquad \gamma<10^{-3}.
\end{equation}
Then choose a constant $b\in(0,1)$ sufficiently small, then a rounding
constant $C_{\mathrm{rnd}}$ sufficiently large, and finally $K_0$
sufficiently large, each in terms of all preceding choices.

For $r\ge2$ and an integer $d\ge3$, define
\begin{equation}\label{eq:theta}
 \theta_{r,d}
 =\frac{b}{
 r^{d/(d-1)}d^{2d/(d-1)}
 (\log(2rd))^{2/(d-1)}}
\end{equation}
and
\begin{equation}\label{eq:T}
 \mathcal T_{r,d}
 =r^{(2d^2-1)/(d-1)^2}
  d^{2d(2d-1)/(d-1)^2}
  (\log(2rd))^{2d^2/(d-1)^2}.
\end{equation}
We first prove the following version of \cref{thm:offdiag}, keeping the
order $d$ of the root filter as a parameter.

\begin{proposition}\label{thm:offdiag-flex}
There are absolute constants $c,C>0$ such that the following holds.  Let
$r\ge2$ and let $d$ be an integer with
\[
 3\le d\le\max\{3,\log(2r)\}.
\]
Then there is a constant $A_{r,d}\le C\theta_{r,d}\mathcal T_{r,d}$ such that,
for every positive integer vector $\kk=(k_1,\ldots,k_r)$,
\[
 R(\kk)\le
 \left\lceil
 \exp\bigl(\Ent(\kk)-c\theta_{r,d}\pot(\kk)+A_{r,d}\bigr)
 \right\rceil.
\]
\end{proposition}

We now choose the parameters used in the induction.  Put
\begin{equation}\label{eq:Gamma-eta}
 \Gamma=\gamma\theta_{r,d},
 \qquad
 \eta=\frac{16\Gamma}{r}.
\end{equation}

\begin{lemma}\label{lem:parameters}
There is an absolute constant $K_0$ such that the following holds.  Suppose
$\pot(\kk)>0$ and
\[
 m=\min_i k_i\ge m_0:=K_0\mathcal T_{r,d}.
\]
Let $q_i=k_i/B$ and define
\begin{equation}\label{eq:t-sstar}
 t=\lfloor\theta_{r,d}m\rfloor,
 \qquad
 s_\ast=\lceil4\Gamma rm\rceil,
\end{equation}
\[
 p_i=q_i-2\eta,
 \qquad
 \delta_i=\zeta p_i\theta_{r,d},
 \qquad
 L_i=\log(1/\delta_i),\quad L=\max_iL_i,
 \qquad
 \lambda_0=\frac{AL^2}{\theta_{r,d}}.
\]
Then the hypotheses in \eqref{eq:book-parameters} hold, and
\begin{equation}\label{eq:p-window}
 \frac1{3r}\le p_i\le\frac2r,
 \qquad
 L_i=\Theta(\log(2rd))\quad\text{uniformly in $i$},
\end{equation}
\begin{equation}\label{eq:Pi-bound}
 \Pi_i\le\left(3\zeta+\frac6A\right)\theta_{r,d}.
\end{equation}
Moreover, if $\Xi$ is as in \cref{thm:book}, then
\begin{equation}\label{eq:Xi-bound}
 rt\Xi+\log(2rt)\le\tfrac1{16}\,rm\log(2r).
\end{equation}
Finally, $\theta_{r,d}^2m\ge C_{\mathrm{rnd}}$ once $K_0$ is sufficiently
large.
\end{lemma}

\begin{proof}
Since $\pot(\kk)>0$, we have $M<2m$.  Hence
\begin{equation}\label{eq:q-window}
 \frac1{2r}<q_i<\frac2{r+1}
 \qquad(i\in[r]).
\end{equation}
The smallness of $\gamma$ and $b$ gives the bounds on $p_i$ in
\eqref{eq:p-window} and, because $\delta_i=\zeta p_i\theta_{r,d}$, also the
uniform estimate $L_i=\Theta(\log(2rd))$.
The only nontrivial condition in \eqref{eq:book-parameters} is
\[
 t\ge\lambda_0(\min_i\delta_i)^{-1/(d-1)}.
\]
Using \eqref{eq:p-window}, it is enough that
\[
 m\ge C'
 (\log(2rd))^2r^{1/(d-1)}
 \theta_{r,d}^{-(2d-1)/(d-1)}
\]
for a suitable absolute constant $C'$.
Substituting \eqref{eq:theta} into the right-hand side gives, up to an absolute
factor,
\[
 r^{(2d^2-1)/(d-1)^2}
 d^{2d(2d-1)/(d-1)^2}
 (\log(2rd))^{2d^2/(d-1)^2}
 =\mathcal T_{r,d},
\]
so the condition follows from $m\ge K_0\mathcal T_{r,d}$.
The same substitution gives
\[
 \theta_{r,d}^2\mathcal T_{r,d}
 =b^2 r^{(2d-1)/(d-1)^2}
 d^{2d/(d-1)^2}
 (\log(2rd))^{2+2/(d-1)^2},
\]
which is bounded below by a constant multiple of $b^2\log^2(2rd)$.
Thus $\theta_{r,d}^2m$ can be made uniformly large by increasing $K_0$.
For \eqref{eq:Pi-bound}, use $\delta_i=\zeta p_i\theta_{r,d}$ and
$\log(1/p_i)\le L_i$ to obtain
\[
 \Pi_i
 =3\zeta\theta_{r,d}
 +\frac{6L_i\log(1/p_i)}{AL^2/\theta_{r,d}}
 \le\left(3\zeta+\frac6A\right)\theta_{r,d}.
\]
It remains to bound $\Xi$.  Put $\tau=(d-1)/d$ and
$\ell=\log(2rd)$.  From \eqref{eq:theta},
\[
 \theta_{r,d}^{\tau}rd^2\ell\,L^{2/d}
 =O(b^\tau\ell).
\]
Since $C_{r,d}\le C_0rd^2\ell$ and $\lambda_0=AL^2/\theta_{r,d}$, this bounds
each of the two correlation terms of $\Xi$ by
\[
 O\!\left(\frac{b^\tau\log(2r)}{\theta_{r,d}}\right).
\]
Also $\rho=O(r+\log r)$.  Multiplying by
$rt\le r\theta_{r,d}m$ and using that $b$ was chosen sufficiently small gives
\eqref{eq:Xi-bound}; the term $\log(2rt)$ is absorbed because
$m\ge K_0\mathcal T_{r,d}$.
\end{proof}

We shall also need a lower bound on the reservoir produced by weighted
regularisation.

\begin{lemma}\label{lem:reservoir}
Under the hypotheses and notation of \cref{lem:parameters}, let $N$ satisfy
\[
 \log N\ge \Ent(\kk)-\Gamma\pot(\kk),
\]
and apply \cref{lem:weighted} with weights $q_i=k_i/B$ and stopping size
$s_\ast$.  If the procedure stops with $s<s_\ast$, then, provided $K_0$ is
sufficiently large,
\begin{equation}\label{eq:reservoir-conclusion}
 \eta|W|\ge1,
 \qquad
 |W|\ge2rt\e^{rt\Xi},
 \qquad
 |W|\ge\max_i p_i^{-t}\e^{\Pi_it}.
\end{equation}
\end{lemma}

\begin{proof}
Since $M<2m$, one has $B\ge rm$ and $k_i/B\ge1/(2r)$.  Moreover
$B\ge M+(r-1)m>M(r+1)/2$, so that
\[
 \Ent(\kk)=\sum_ik_i\log\frac B{k_i}
 \ge B\log\frac BM
 \ge rm\log\frac{r+1}2,
\]
and since $\log\frac{r+1}2\ge\tfrac14\log(2r)$ for every $r\ge2$
(equivalently, $(r+1)^4\ge32r$),
\[
 \Ent(\kk)\ge\tfrac14\,rm\log(2r).
\]
From \eqref{eq:weighted-size}, $s<s_\ast$ and $q_i\ge1/(2r)$,
\begin{align*}
 \log|W|
 &\ge \log N+s\log(1+\eta)+\sum_i s_i\log q_i\\
 &\ge \Ent(\kk)-\Gamma m-s_\ast\log(2r).
\end{align*}
Here $\pot(\kk)\le m$.  Since
$s_\ast=O(\Gamma rm)=O(\theta_{r,d}rm)$ and both $\gamma$ and $b$ were chosen
sufficiently small, the subtracted terms total at most
$\tfrac18rm\log(2r)$, and hence
\[
 \log|W|\ge\tfrac18\,rm\log(2r).
\]
The second inequality in \eqref{eq:reservoir-conclusion} now follows from
\eqref{eq:Xi-bound}, and the first because $K_0$ is sufficiently large.  Finally, by \eqref{eq:p-window} and \eqref{eq:Pi-bound},
\[
 t\log(1/p_i)+\Pi_it
 \le O(\theta_{r,d}m\log(2r)),
\]
which is again dominated by the lower bound for $\log|W|$ just obtained.
This proves the last inequality.
\end{proof}

\begin{proof}[Proof of \cref{thm:offdiag-flex}]
Fix $r,d$ and the constants above, and write
$\theta=\theta_{r,d}$.  Let
\[
 m_0=K_0\mathcal T_{r,d},
 \qquad
 A_{r,d}=\Gamma m_0.
\]
Since $\Gamma=\gamma\theta$, this satisfies the bound
$A_{r,d}\le C\theta_{r,d}\mathcal T_{r,d}$ of \cref{thm:offdiag-flex} with
$C=\gamma K_0$, and the saving coefficient there is $c=\gamma$.  We prove by
induction on $B=\sum_i k_i$ that
\begin{equation}\label{eq:induction}
 R(\kk)\le
 \left\lceil\exp\bigl(\Ent(\kk)-\Gamma\pot(\kk)+A_{r,d}\bigr)\right\rceil.
\end{equation}
If $\pot(\kk)=0$, this follows from \eqref{eq:ES}.  If
$m=\min_i k_i\le m_0$, then $\pot(\kk)\le m_0$ and the additive term again
makes \eqref{eq:induction} weaker than \eqref{eq:ES}.  We may therefore assume
\[
 \pot(\kk)>0,
 \qquad m>m_0.
\]
We record once how integer roundings are absorbed.  For every positive
integer vector $\mathbf x$ we have $\pot(\mathbf x)\le\min_ix_i$, while
writing $\Ent(\mathbf x)=\sum_ix_i\log\bigl(\sum_jx_j/x_i\bigr)$ shows that
the summand corresponding to a minimal coordinate is already at least
$(\min_ix_i)\log r$; since $\Gamma\le\gamma b<\log2$, every exponent
appearing in \eqref{eq:induction} is therefore nonnegative, and each ceiling
is at most twice the corresponding exponential.  Consequently, a margin of
$\log2$ in an exponent inequality below absorbs all integer roundings.  As
$\gamma$, $\zeta$ and $A$ are already fixed, we may choose
$C_{\mathrm{rnd}}$, and then $K_0$, so large that the margins obtained in
\eqref{eq:long-margin} and \eqref{eq:page-positive} below are at least
$\log2$; we use this without further comment.
Let $N$ be the right-hand side of \eqref{eq:induction}, and consider an
arbitrary $r$-colouring of $K_N$.  Apply \cref{lem:weighted} with
$q_i=k_i/B$, the parameter $\eta$ from \eqref{eq:Gamma-eta}, and the stopping
size $s_\ast$ from \eqref{eq:t-sstar}.  Let $S_i,W$ be the resulting sets and
write $s_i=|S_i|$, $\ssv=(s_1,\ldots,s_r)$ and $s=\sum_i s_i$.  If $s_i\ge k_i$ for some $i$ then we are
done, so assume $s_i<k_i$ for every $i$.

Suppose first that $s=s_\ast$.  Put $\cc=\kk-\ssv$ and
$\DG=\DG_{\kk}(\ssv)$.  By \eqref{eq:weighted-size} and
\cref{lem:entropy},
\begin{align*}
 \log|W|
 &\ge \Ent(\kk)-\Gamma\pot(\kk)+A_{r,d}
      +\sum_i s_i\log q_i+s\log(1+\eta)\\
 &=\Ent(\cc)-\Gamma\pot(\cc)+A_{r,d}
   +\DG+s\log(1+\eta)
   -\Gamma\bigl(\pot(\kk)-\pot(\cc)\bigr).
\end{align*}
Using \cref{lem:potential},
$\log(1+\eta)\ge\eta/2=8\Gamma/r$, and
$s=s_\ast\ge4\Gamma rm$, we obtain
\begin{align}
 &\DG+s\log(1+\eta)
 -\Gamma\bigl(\pot(\kk)-\pot(\cc)\bigr)\notag\\
 &\hspace{2cm}\ge
 \frac12\DG+\frac{4\Gamma s}{r}-10\Gamma^2m
 \ge6\Gamma^2m.\label{eq:long-margin}
\end{align}
Since $6\Gamma^2m\ge6\gamma^2C_{\mathrm{rnd}}\ge\log2$ by
\cref{lem:parameters} and the choice of $C_{\mathrm{rnd}}$, the rounding
convention gives $|W|\ge R(\cc)$ by the induction hypothesis.  Hence $W$
contains, for some $j$, a colour-$j$ clique of size $k_j-s_j$, which
together with $S_j$ completes the target $K_{k_j}$.

We may therefore assume $s<s_\ast$.  The minimum-degree conclusion
\eqref{eq:weighted-degree} holds, and \cref{lem:reservoir} gives
\eqref{eq:reservoir-conclusion}.  In particular, for every $w\in W$,
\[
 |N_i(w)\cap W|
 \ge(q_i-\eta)|W|-1
 \ge(q_i-2\eta)|W|=p_i|W|.
\]
For each colour $i$ with $s_i+t<k_i$, let
\[
 \cc^{(i)}=\kk-\ssv-t\ee_i,
 \qquad m_i=R(\cc^{(i)}).
\]
If $s_i+t\ge k_i$, set $m_i=1$; in that case a colour-$i$ spine of size
$t$ already completes the target.
Fix a colour $i$ with $s_i+t<k_i$.  Put
$\uu=\ssv+t\ee_i$, $U=s+t$, and
$\DG_i=\DG_{\kk}(\uu)$.  By the induction hypothesis,
\eqref{eq:entropy-identity} and the rounding convention, the page condition
$|W|\ge p_i^{-t}\e^{\Pi_it}m_i$ follows once
\begin{align*}
 \mathcal M_i:={}&
 \DG_i+s\log(1+\eta)
 -\Gamma\bigl(\pot(\kk)-\pot(\cc^{(i)})\bigr)\\
 &\qquad -t\log(q_i/p_i)-\Pi_it\ge\log2.
\end{align*}
In the notation of \cref{lem:entropy}, $z_i=s_i+t-q_i(s+t)$.  From
\eqref{eq:q-window},
$s<s_\ast\le5\Gamma rm$ (note that $\Gamma rm\ge\gamma C_{\mathrm{rnd}}\ge1$),
$t\ge\theta m/2$, and $\gamma<10^{-3}$, we have
\[
 z_i\ge(1-q_i)t-q_is\ge\frac t6.
\]
Since $k_i<2m$, \eqref{eq:entropy-lower} gives
\[
 \DG_i\ge\frac{z_i^2}{2k_i}\ge\frac{t^2}{144m}.
\]
Applying \cref{lem:potential} to the potential term of $\mathcal M_i$, and
using $s\log(1+\eta)\ge4\Gamma s/r$, gives
\[
 \mathcal M_i\ge
 \frac12\DG_i-\frac{4\Gamma t}{r}-10\Gamma^2m
 -t\log(q_i/p_i)-\Pi_it.
\]
Because $t\ge\theta m/2$,
\[
 \frac12\DG_i\ge\frac{\theta t}{576}.
\]
Moreover, $q_i\ge1/(2r)$ and $p_i=q_i-2\eta$ imply
\[
 \log(q_i/p_i)
 =-\log\left(1-\frac{2\eta}{q_i}\right)
 \le128\gamma\theta,
\]
and $\Pi_it\le(3\zeta+6/A)\theta t$ by \eqref{eq:Pi-bound}.  Thus
\begin{equation}\label{eq:page-positive}
 \mathcal M_i\ge
 \theta t\left(
 \frac1{576}-130\gamma-20\gamma^2-3\zeta-\frac6A
 \right)\ge\log2
\end{equation}
by \eqref{eq:constant-choice}, since
$\theta t\ge\theta^2m/2\ge C_{\mathrm{rnd}}/2$ and $C_{\mathrm{rnd}}$ is
sufficiently large.  Hence
\[
 |W|\ge p_i^{-t}\e^{\Pi_it}m_i
\]
for every colour $i$ with $s_i+t<k_i$.  If $s_i+t\ge k_i$, the same
inequality, now with $m_i=1$, is the last conclusion of \cref{lem:reservoir}.
We may now apply \cref{thm:book} with
$X=Y_1=\cdots=Y_r=W$.  It gives a colour $i$, a colour-$i$ spine $T$ of size
$t$, and a page set $P$ of size $m_i$.  If $s_i+t\ge k_i$, then
$S_i\cup T$ already contains a colour-$i$ $K_{k_i}$.  Otherwise
$|P|=R(\cc^{(i)})$, so $P$ contains, for some $j$, a colour-$j$ clique
whose size is the $j$th coordinate of $\cc^{(i)}$.  If $j=i$, join it to
$S_i\cup T$; if $j\ne i$,
join it to $S_j$.  In either case a target $K_{k_j}$ is completed.  This
proves \eqref{eq:induction}, and hence the proposition.
\end{proof}

\begin{proof}[Proof of \cref{thm:offdiag}]
It remains to choose the order of the root filter.  Let
\[
 d=\max\left\{3,\left\lceil\frac12\log(2r)\right\rceil\right\}.
\]
Then $d=\Theta(\log(2r))$, so the factors $r^{1/(d-1)}$, $d^{1/(d-1)}$ and
$(\log(2rd))^{1/(d-1)}$ are bounded above and below by absolute constants.
Substituting into \eqref{eq:theta} and \eqref{eq:T} therefore gives
\[
 \theta_{r,d}=\Theta\!\left(\frac1{r\log^2(2r)}\right),
 \qquad
 \mathcal T_{r,d}=O(r^2\log^6(2r)),
\]
and hence $A_{r,d}\le C\theta_{r,d}\mathcal T_{r,d}=O(r\log^4(2r))$.  With
this choice of $d$, \eqref{eq:offdiag} follows from \cref{thm:offdiag-flex}.
\end{proof}

\begin{proof}[Proof of \cref{thm:main}]
Apply \cref{thm:offdiag} to the diagonal vector $(k,\ldots,k)$.  Since
$\Ent(k,\ldots,k)=rk\log r$ and $\pot(k,\ldots,k)=k$, the additive term in
\eqref{eq:offdiag} is absorbed into half of the saving whenever
$k\ge K r^2\log^6(2r)$, after increasing the absolute constant $K$.
\end{proof}

\cref{thm:offdiag} also gives the following asymptotic consequence.

\begin{corollary}\label{cor:offdiag}
There is an absolute constant $c>0$ such that the following holds for every
$r\ge2$.  Let $\boldsymbol\alpha=(\alpha_1,\ldots,\alpha_r)\in\R_{>0}^r$
satisfy $\alpha_{\max}<2\alpha_{\min}$.  Then
\[
 \limsup_{n\to\infty}\frac1n
 \log R(\lceil\alpha_1n\rceil,\ldots,\lceil\alpha_rn\rceil)
 \le \Ent(\boldsymbol\alpha)
 -c\,\frac{2\alpha_{\min}-\alpha_{\max}}{r\log^2(2r)},
\]
where $\Ent(\boldsymbol\alpha)$ is defined by the homogeneous extension of
\eqref{eq:entropy-def-intro}.
\end{corollary}

\begin{proof}
Let $\kk^{(n)}=(\lceil\alpha_1n\rceil,\ldots,\lceil\alpha_rn\rceil)$.
Apply \cref{thm:offdiag} to $\kk^{(n)}$, divide by $n$, and let $n\to\infty$.
By homogeneity and continuity,
\[
 \Ent(\kk^{(n)})=n\Ent(\boldsymbol\alpha)+o(n),
 \qquad
 \pot(\kk^{(n)})=(2\alpha_{\min}-\alpha_{\max})n+O(1),
\]
which gives the result.
\end{proof}

Thus the recursive argument improves the classical multinomial exponent
throughout a neighbourhood of the diagonal, not only on the diagonal
itself.

\section*{AI Usage Disclosure}

Anthropic's Claude Fable 5 and OpenAI's GPT-5.6 Sol were used for
exploratory discussions during the development of the proofs and to assist
with drafting and revising the manuscript.  The author takes full
responsibility for the paper's contents and correctness.

\end{document}